\documentclass[12pt]{article}
\usepackage{amsmath,amssymb,bbm,amsthm,amsfonts,amscd}
\usepackage[numeric]{amsrefs}
\input colordvi

\swapnumbers

\newtheorem{thm}[equation]{Theorem}

\newtheorem{lem}[equation]{Lemma}

\newtheorem{prop}[equation]{Proposition}

\DeclareMathOperator{\tr}{tr}

\numberwithin{equation}{section}

\newcommand\G{\Gamma}
\newcommand\f{\frac}

\newcommand{\Z}{{\mathbb{Z}}}

\newcommand\F{{\mathbb F}}

\renewcommand\i{^{-1}}
\renewcommand\({\left(}
\renewcommand\){\right)}

\newcommand{\gobble}[1]{}
  \newcommand{\rangeref}[2]{%
    \ref{#1}--\afterassignment\gobble\fam 0\ref{#2}%
  }

\begin{document}

\title{Non-degeneracy of Pollard Rho Collisions}

\date{August 31, 2008}
\author{Stephen D. Miller\thanks{Partially supported by NSF grant
DMS-0601009 and an Alfred P. Sloan Foundation Fellowship.
}~\ and\  Ramarathnam Venkatesan}

\maketitle

\begin{abstract}

The Pollard $\rho$ algorithm is a widely used algorithm for solving discrete logarithms on general cyclic groups, including elliptic curves.
Recently the first nontrivial runtime estimates were provided for it, culminating in a sharp $O(\sqrt{n})$ bound for the collision time on a cyclic group of order $n$ \cite{kmpt2}.  In this paper we show that for $n$ satisfying a mild arithmetic condition, the collisions guaranteed by these results are nondegenerate with high probability:~that is, the Pollard $\rho$ algorithm successfully finds the discrete logarithm.

\vspace{.2 cm}

\end{abstract}

\section{Introduction}

The Pollard $\rho$ algorithm is, to date, the leading algorithm for solving discrete logarithm problems on general groups, including  elliptic curves.  The algorithm can be stated as follows.  Let $G$ be a cyclic group of order $n$ generated by the element $g$;~$n$ may assumed to be a large prime because of the Pohlig-Hellman reduction \cite{hellmanpohlig}.  Let $h=g^y$ be the element whose discrete logarithm $y\neq 1$ (unknown) is to be found, and let $x_0=h$ or a random power $g^{r_1}h^{r_2}$ (which turns out to be only slightly less general).  Let $G=S_1 \cup S_2 \cup S_3$ be  a random partition of $G$ into three disjoint subsets, in which each element has a 1/3 probability of belonging to each $S_j$.\footnote{In practice, the assignment is accomplished using a hash function which is expected to behave randomly, as storing the partitions themselves would take up too much memory.  A formal model would assume a cryptographically strong pseudo-random function whose underlying cryptographic primitive would have a security estimate   exceeding the runtime of the Pollard Rho algorithm.  Such implementation details  to justify the random walk model  used in our analysis are well understood.}
Define an iteration $x_{k+1}=f(x_k)$, where
\begin{equation}\label{iteration}
    f(x) \ \ = \ \ \left\{
                     \begin{array}{ll}
                       gx\,, &  \ x\,\in\,S_1\,; \\
                       hx\,, & \ x\,\in\,S_2\,; \\
                       x^2\,, & \ x\,\in\,S_3\,.
                     \end{array}
                   \right.
\end{equation}
At each stage $x_k$ may be written as $g^{a_k y+b_k}$, where the coefficients $a_k,b_k\in \Z/n\Z$ are known.  Iterate until a collision of values $x_k=x_\ell$ has been found, and if the collision is ``non-degenerate'' (meaning $(a_k,b_k)\neq (a_\ell,b_\ell)$), solve for the discrete logarithm using the formula $y=\f{b_\ell-b_k}{a_k-a_\ell}$.

The algorithm is conjectured to run in time $O(\sqrt{n})$ with high probability.  It is the only such algorithm which uses small memory and which works for general groups.  Though faster algorithms are known for specific incarnations of cyclic groups\footnote{For example, index calculus provides a subexponential algorithm on the group $\F_p^*$, which is abstractly isomorphic to a cyclic group of order $n=p-1$.  Note that this is {\em not} itself an example of a prime order cyclic group as treated above:~one must apply the Pohlig-Hellman reduction first.}, a  theorem of Victor Shoup \cite{shoup} asserts that no algorithm on a general group can be faster -- aside from improving the implied multiplicative constant.  For it to be successful, two things must happen:
\begin{enumerate}
                                     \item A collision must be found in time $O(\sqrt{n})$.
                                     \item This collision must be non-degenerate.
                                   \end{enumerate}
Item 1 has been the subject of a number of recent papers, before which there were no nontrivial bounds on the runtime at all.  First, a collision time of $O(\sqrt{n}(\log n)^{3})$ was  shown in \cite{prho}, which was successively improved by  \cite{kmpt1} and \cite{kmpt2} to the optimal $O(\sqrt{n})$ bound.

The purpose of this paper is to address Item 2 for the Pollard $\rho$ algorithm (it is, however, settled for some variants of Pollard $\rho$, as in \cite{horwen}). Unfortunately as of yet we are unable to make the result unconditional, for it depends on the multiplicative order of 2 modulo $n$ (the least positive integer $k$ such that $2^k\equiv 1\pmod n$).
We prove the following result, which is a complete runtime analysis for almost all group orders $n$:

\begin{thm}\label{mainthm}  Consider the Pollard $\rho$ algorithm as above on a group $G=\langle g \rangle$ of prime order $n$, starting at a random point $x_0=g^{r_1}h^{r_2}$. Suppose that the multiplicative order of 2 modulo $n$ is at least $c_0(\log n)^3$, where $c_0$ is the absolute constant coming from Proposition~\ref{hittingsets}.  Then any Pollard $\rho$ collision occurring before time $T$ is nondegenerate with probability at least $1-\f 32 \f{T^2}{n^2}$.  In particular, the collisions guaranteed by \cite{kmpt2} to occur with high probability within time $O(\sqrt{n})$ are nondegenerate with probability at least $1-O(\f 1n)$.
\end{thm}

\noindent {\bf Remarks:}  1) Though the probability of nondegeneracy is heuristically higher than that of  collisions, in practice it has been much more difficult to prove nondegeneracy.

2) The multiplicative order of 2 modulo $n$ is typically quite large, e.g.~it equals $n-1$ if 2 generates $(\Z/n\Z)^*$, which it frequently does.   There do exist primes  with multiplicative order the size of $\log n$ (e.g.~Fermat and Mersenne primes), but those disobeying the condition in the theorem are quite rare.  Indeed, we show in Lemma~\ref{afewbadprimes} that at most $O((\log X)^5)$ such primes $p$ exist in the range $X\le p \le 2X$.

3)  Even if $2$ has small multiplicative order modulo $n$, there is always a prime  $\ell=O((\log n)^6(\log \log n))$ which has multiplicative order at least $c_0(\log n)^3$.  This is because a cyclic group has at most $L^2$ elements of order $\le L$.  (We thank the referee for supplying this argument.)
  If the Pollard $\rho$ algorithm is modified to replace the squaring step by $x\mapsto x^\ell$ instead, the analysis here and in \cite{prho,kmpt1,kmpt2} applies and gives a completely rigorous proof of the same $O(\sqrt{n})$ runtime, with the same $1-O(\f 1n)$ success rate.

4) The reason we need to assume a random starting point, unlike in \cite{prho}, is that we cannot rule out degeneracies in collisions occurring within the first few steps.  Lemma~\ref{tracebd}, in particular, applies only to random starting points.  Once the algorithm has proceeded for $c_0(\log n)^3$ steps a random point is reached regardless of the starting point, but we cannot guarantee a random position before then.

The strategy of the proof starts with the viewpoint  that the Pollard $\rho$ iteration can be modeled as a pseudo-random walk on the ``Pollard $\rho$ graph'':~the graph whose vertices are elements of $G$, and whose (directed) edges have the form
\begin{equation}\label{prhograph}
    x \ \longrightarrow \ xg, \,xh, \ \text{or} \  x^2.
\end{equation}
Indeed, until a collision occurs the iteration by (\ref{iteration}) is in fact a random walk, because the destination from a vertex $x$ depends only on its random assignment to one of the $S_j$; however, it is important that the walk is no longer random after this point, for it enters a loop.  The coefficients $(a_k,b_k)\in(\Z/n\Z)^2$ meanwhile likewise can be modeled as a random walk (until the time of collision) on the following ``Pollard $\rho$ coefficient graph'':
\begin{equation}\label{prhocoeffgraph}
    (a,b) \ \longrightarrow \ (a+1,b), \,(a,b+1),  \ \text{or} \ (2a,2b).
\end{equation}
This graph maps onto the graph (\ref{prhograph}) by $(a,b)\mapsto g^{ay+b}$, where $y$ is the (secret and unknown) exponent of $h=g^y$.

Our argument has two main ingredients.  The first is a spectral upper bound on the mixing time of this graph, which roughly speaking shows that the coefficients $(a_k,b_k)$ become equidistributed after a small number of steps.  This is very similar to the argument in \cite{prho} to guarantee collisions among the $x_k$.  That alone, however, is not enough to show nondegeneracy:~it is important to note that {\em undirected} 3-regular graphs can have this equidistribution feature, while simultaneously having $(a_k,b_k)$ equaling $(a_{k+2},b_{k+2})$ with probability $\ge 1/3$ (for example, going backwards on the edge just traveled).  The second ingredient, an estimate on the number of short cycles, handles this.  It is this part which depends on the condition on the multiplicative order of 2 modulo $n$, and hence which is not completely general.

We conclude this section with the proof of Theorem~\ref{mainthm}, which depends on estimates of the last two sections.  Section 2, roughly speaking, deals with long random walks, while Section 3 with short random walks.  The condition on the multiplicative order of 2 modulo $n$ is needed to make sure their intervals of applicability overlap.

\begin{proof}[Proof of Theorem~\ref{mainthm}]
Once a collision occurs, all future collisions are nondegenerate if and only if the first one was; this is because of the invertibility of the steps in (\ref{prhocoeffgraph}).  Thus, it suffices to assume that no collision has occurred until time $T$, which allows us to model the coefficients $(a_k,b_k)$, $k\le T$, using a random walk on (\ref{prhocoeffgraph}).
Because the starting point $x_0=g^{r_1}h^{r_2}$ is uniformly distributed and walk up to time $T$ is random, the values of each $(a_k,b_k)$ are themselves uniformly distributed.
We show in Proposition~\ref{hittingsets} and Lemma~\ref{tracebd} that for any $m>0$ and a random point $(a,b)\in(\Z/n\Z)^2$,  a random walk of length $m$ starting at $(a,b)$  ends at $(a,b)$  with probability at most  $\f{3}{2}\f{1}{n^2}$.  By the union bounds, the probability of a degeneracy $(a_k,b_k)=(a_\ell,b_\ell)$ occurring for some distinct $k,\ell \le T$  is bounded above by
\begin{equation}\label{pf1}
    \sum_{\scriptstyle{\stackrel{k,\ell \le T}{k\neq \ell}}} P[(a_k,b_k)=(a_\ell,b_\ell)] \ \ \le \ \
     \sum_{\scriptstyle{\stackrel{k,\ell \le T}{k\neq \ell}}}  \f{3}{2}\f{1}{n^2} \ \ < \ \ \f{3}{2}\f{T^2}{n^2}\,.
\end{equation}
\end{proof}

It is a pleasure to acknowledge Ravi Montenegro, Ze'ev Rudnick, Adi Shamir, and Prasad Tetali for their helpful discussions.  We  also thank Curt McMullen for helpful comments concerning the remark at the end of Section~\ref{sec:mix}, and the anonymous referee for their  suggestions for improving the paper, in particular  Remark 3 above.

\section{Mixing time estimates}\label{sec:mix}

Let $A$ denote the adjacency operator of the graph (\ref{prhocoeffgraph}):~it is defined on complex-valued functions $f$ on $(\Z/n\Z)^2$ by the formula
\begin{equation}\label{Aongraph}
    Af(a,b) \ \ = \ \ f(2a,2b) \ + \ f(a+1,b) \ + \ f(a,b+1)\,.
\end{equation}
Such functions themselves form a complex vector space of dimension $n^2$, which is equipped with the usual inner product and norm
\begin{equation}\label{L2V}
    \langle f_1,f_2 \rangle \ \ = \ \ \sum_{a,b\,\in\, \Z/n\Z} f_1(a,b)\,\overline{f_2(a,b)} \  \, , \ \ \ \ \|f\|^2 \ \ = \ \ \langle f,f \rangle\,.
\end{equation}
Of special interest to us is the restriction of $A$ to $\mathbbm{1}^\perp$, the orthogonal complement of the constant function $\mathbbm{1}$ (the functions on the graph whose average value is zero).
The following result relates the operator norm properties of this restriction of $A$, to the mixing properties of the random walk on the graph:

\begin{lem}\label{mixlem}(\cite[Lemma 2.1]{prho})
Let $\G$ denote a directed graph on the vertex set $V$, having both $d$ directed edges entering and exiting each vertex (including multiplicity).
Suppose that there exists a constant $\mu<d$ such that $\| A f \|
\le \mu\| f \|$ for all $f\in \mathbbm{1}^\perp$. Let $S$ be an arbitrary subset of $V$.  Then the
number of paths of length $r \ge \f{\log(2n)}{\log(d/\mu)}$ which
start from any given vertex and end in $S$ is between $\f12 d^r
\f{|S|}{|V|}$ and $\f32 d^r \f{|S|}{|V|}$.
\end{lem}

Thus   sufficiently long random walks hit a set $S$ with probability between $\f{1}{2}\f{|S|}{|V|}$ and $\f{3}{2}\f{|S|}{|V|}$, independent of their starting point.
Unfortunately this Lemma does not apply directly to our situation, because it can happen that $\|A f\|=\|f\|$ for some functions $f$.  However, to show that random walks mix it suffices to work two steps at a time; fortunately, a nontrivial operator norm estimate applies to $A^2$ instead, which corresponds to the adjacency operator for the graph on $(\Z/n\Z)^2$ with edges
\begin{multline}\label{rho2step}
    (a,b) \ \longrightarrow (4a,4b),\,(2a+1,2b),\,(2a,2b+1),
    \,(2a+2,2b),\,(2a,2b+2),\\ (a+1,b+1),\,(a+1,b+1),\,
    (a+2,b), \ \text{or} \ (a,b+2).   \ \
\end{multline}

\begin{prop}\label{spectralbd}
With $A$ denoting the adjacency operator of the graph (\ref{prhocoeffgraph}) and the standing assumption that $n$ is an odd prime, there exists an absolute constant $c>0$ such that
\begin{equation}\label{a2bd}
    \|A^2f\| \ \ \le \ \ \(3-\f{c}{(\log n)^2}\)^2\|f\| \  , \ \ f \,\in\,\mathbbm{1}^\perp.
\end{equation}
\end{prop}
\begin{proof}

Any function $f$ on the vertices may be expanded in terms of the additive
characters $\chi_{k,\ell}(x,y)=e^{2\pi i (k x+\ell y)/n}$:
\begin{equation}\label{characterexpansion}
    f \ =  \ \sum_{k,\ell\,\in\, \Z/n\Z} \,c_{k,\ell}\,
\chi_{k,\ell}\,.
\end{equation}
The condition that $f\in \mathbbm{1}^\perp$ is equivalent to
$c_{0,0}=0$.  The action of $A$ on the character $\chi_{k,\ell}$ is
given by
\begin{equation}\label{Aonchar}
    A\,\chi_{k,\ell} \ = \ d_{k,\ell}\, \chi_{k,\ell} \, + \,
    \chi_{2k,2\ell}\,,
\end{equation}
where
\begin{equation}\label{dkldef}
\ \, \ \ \ d_{k,\ell} \
    = \ e^{2\pi i k/n}+e^{2\pi i \ell/n}\,.
\end{equation}
Thus $A$ is the sum of
the diagonal operator
$D:\chi_{k,\ell}\mapsto d_{k,\ell}\,\chi_{k,\ell}$ and the
permutation operator $P:\chi_{k,\ell}\mapsto \chi_{2k,2\ell}$.  The adjoint of $A$ under the inner product (\ref{L2V})   is $A^*=\overline{D}+P\i$.   Let us write
\begin{equation}\label{adjcalc}
\aligned
    A^{*^2}A^2 \ \ = &  \ \ (\overline{D}^2 + P\i \overline{D} + \overline{D} P\i +
    P^{-2})(D^2+PD+DP+P^2) \\
    = & \ \ X_1 \ + \ X_2\,,
\endaligned
\end{equation}
where $X_1=\overline{D}^2PD+\overline{D}P$, and $X_2$ is the remaining sum of 14 terms from the expansion of the first line.  Because  $|d_{k,\ell}|=2|\cos(\f{\pi(k-\ell)}{n})|\le 2$ and in fact equals 2 when $k=\ell$, the operator norms of
$D$ and $\overline{D}$ are $\|D\|=\|\overline{D}\|=2$. Likewise $\|P\|=\|P\i\|=1$, because $P$ preserves norms.  It follows from the sum of 14 terms defining $X_2$ that $\|X_2\|\le 71$.
Using this fact and Cauchy-Schwartz, we get the bound
\begin{equation}\label{adjcalc2}
\aligned
    \|A^2f\|^2 \ \ = & \ \ \langle f,A^{*^2}A^2f \rangle \ \le \
    71 \, \|f\|^2 \  + \  |\langle f,(\overline{D}^2PD+\overline{D}P)f\rangle|\,.
\endaligned
\end{equation}
In order to prove (\ref{a2bd}) it now suffices to show the bound
\begin{equation}\label{adjcalc25}
\aligned
 |\langle f,(\overline{D}^2PD+\overline{D}P)f \rangle| \ \ \le &  \ \ \(10-\f{c}{(\log n)^{2}}\)\, \|f\|^2 \ \ \\ =&  \ \ \(10-\f{c}{(\log n)^{2}}\)\,n^2\sum_{(k,\ell)\neq(0,0)} |c_{k,\ell}|^2\,,
 \endaligned
\end{equation}
 for some
absolute constant $c>0$.  Here  we have used (\ref{characterexpansion}) as well as  the inner product relation
\begin{equation}\label{innerprodofchar}
    \langle \chi_{k,\ell},\chi_{k',\ell'} \rangle  \ \ = \ \ \left\{
\begin{array}{ll}
                                                                 n^2\,, &
(k,\ell)=(k',\ell') \\
                                                                 0\ , &
\hbox{otherwise.}
                                                               \end{array}
                                                             \right.
\end{equation}
Since
\begin{equation}\label{adjcalc33}
    (\overline{D}^2PD+\overline{D}P)\chi_{k,\ell} \   =  \ \mu_{k,\ell}\,\chi_{2k,2\ell}\ , \ \ \ \mu_{k,\ell} \ = \  \overline{d_{2k,2\ell}}^{\,2}  d_{k,\ell}+\overline{d_{2k,2\ell}}\,,
\end{equation}
we have that
\begin{equation}\label{adjcalc3}
    |\langle f,(\overline{D}^2PD+\overline{D}P)f \rangle| \ \  \le \ \ n^2
\sum_{(k,\ell)\neq(0,0)} \,|c_{k,\ell}|\,|c_{2k,2\ell}|\,|\mu_{k,\ell}|\,.
\end{equation}
We now group the indices  $(k,\ell)\neq (0,0)$ into the $n+1$  lines through the origin in $(\Z/n\Z)^2$.  Using the bounds
\begin{equation}\label{adjcalc4}
\aligned
   | \mu_{k,\ell}| \ \ \le & \ \ 8 \,+\, 2|\cos \textstyle{\f{2(k-\ell)\pi}{n}}|\, , \ \ \ \text{for the lines with~}k\neq \ell\,,
   \\
  | \mu_{k,k}| \ \ \le & \ \ 6 \,+\,4|\cos\textstyle{\f{ k \pi}{n}}  |\,,
   \endaligned
\end{equation}
and the fact that 2 is invertible modulo $n$,
the desired bound (\ref{adjcalc25}) reduces to the estimates
\begin{equation}\label{adjcalc5}
\aligned
     \sum_{k\,=\,1}^{n-1}x_k\,x_{2k} \ \ \le & \ \ \sum_{k\,=\,1}^{n-1}x_k^2 \ \ \ \ \ \ \ \qquad \qquad  \ \qquad \text{and}
    \\
    \sum_{k\,=\,1}^{n-1}x_k\,x_{2k}\,|\cos \textstyle{\f{\pi k}{n}}| \  \ \le & \ \ \(1 \,-\,\f{c'}{(\log n)^2}\)\,\sum_{k\,=\,1}^{n-1}x_k^2\,, \ \text{
    for some absolute $c'>0$,}
\endaligned
\end{equation}
for any real numbers $x_1,\ldots,x_{n-1}$.
The first inequality follows from $x_k x_{2k}\le \f 12 (x_k^2+x_{2k}^2)$, and the second  is proven in
\cite[Prop. 3.1]{prho}.
\end{proof}

Combining these, we have shown

\begin{prop}\label{hittingsets}
There exists an absolute positive constant $c_0$ such that the number of paths of length $r\ge c_0(\log n)^3$ on the graph (\ref{prhocoeffgraph}) which begin at the vertex $(a,b)$ and end at the vertex $(a',b')$ is between $\f 12 \f{3^r}{n^2}$ and  $\f 32 \f{3^r}{n^2}$.  In particular, the probability of a random walk of length $r\ge c_0(\log n)^3$ ending at its starting point is at most $\f{3}{2}\f{1}{n^2}$.
\end{prop}

\noindent {\bf Remarks:} This mixing time estimate for the random walk can also be proven using the method of canonical paths from \cite{montteta}.

Interestingly, the mixing time estimate for the Pollard $\rho$ graph (\ref{prhograph}) in \cite{prho} does not need the step $x\rightarrow xh$: the steps $x\mapsto xg$ and $x\mapsto x^2$ suffice.  This follows from the same method of proof, and is  suggested by the heuristic that the $x\mapsto xh$ step is approximated by the other two.
  However, the Pollard $\rho$ coefficient graph  does not rapidly mix unless all three steps in (\ref{prhocoeffgraph}) are present.  In any case, all three steps are necessary for the execution of the Pollard $\rho$ algorithm.

\section{Trace estimates}\label{sec:trace}

\begin{lem}\label{tracebd}
If $k \ge 1$ is less than the multiplicative order of 2 modulo $n$, then there are precisely $3^k-2^k$ closed cycles on the graph (\ref{prhocoeffgraph}) of length   $k$.  In particular, if $(a,b)$ is a random point in $(\Z/n\Z)^2$, the probability  is at most $\f{1}{n^2}$ that a random walk of length $k$ which starts at $(a,b)$ also ends at $(a,b)$.
\end{lem}

The number of such cycles is also given by $\tr A^k$, where $A$ is the adjacency operator of the graph.  The above estimate, however, does not seem to follow from the spectral techniques of the previous section.

\begin{proof}
Every path involves either doubling the coefficients $(a,b)$, or adding 1 to one of them.  Thus all paths of length $k$ starting from the vertex $(x,y)$ have the form
\begin{equation}\label{tr1}
    T\,: \ (x,y) \ \mapsto 2^s(x,y) \,+(u,v)\,,
\end{equation}
where $s\le k$ equals the number of doubling steps in the path, and $u,v\in \Z/n\Z$ are independent of $x$ and $y$  .  (This characterization obviously holds for $k=1$, and in general by induction.)  A closed cycle is equivalent to a fixed point for $T$.  Of the $3^k$ possible paths starting from $(x,y)$, exactly $2^k$ have $s=0$.  For those walks, $(u,v)\neq (0,0)$ since all steps are of the form $(a,b)\mapsto (a+1,b)$ or $(a,b+1)$.  Thus $T$ has no fixed points in this situation.

However, if $s\neq 0$ then $2^s$ is not congruent to 1 modulo $n$ because of the multiplicative order condition.  In this situation, $T$ has exactly one fixed point.  The closed cycles come from these $3^k-2^k$ cases.
\end{proof}

This concludes the estimates necessary for the proof of Theorem~\ref{mainthm}.
We conclude with the following lemma, which shows that primes for which
2 has multiplicative order smaller than its condition $c_0(\log n)^3$ are extremely rare.

\begin{lem}\label{afewbadprimes}
Let $X>0$, $c>0$, and $B$ the set of primes $p$ in the interval $[X,2X]$ such that the multiplicative order of 2 modulo $p$ is bounded by $c(\log X)^3$.  Then the size of $B$ is bounded by
\begin{equation}\label{sizeofB}
    |B| \ \ \le \ \ \f{c^2 \log(2)}{2} (\log X)^5.
\end{equation}
\end{lem}
\begin{proof}
The condition on $p\in B$ states that $p$ divides $\prod_{k \le c(\log X)^3}(2^k-1)$.  Primality therefore implies that
\begin{equation}\label{fewbad2}
    \prod_{p\in B}p \ \ \ \text{divides} \ \ \prod_{k \,  \le \,  c(\log X)^3}(2^k-1)\,,
\end{equation}
and in particular satisfies
\begin{equation}\label{fewbad3}
    X^{|B|} \ \ \le \ \ \prod_{p\in B}p  \ \ \le \ \ \prod_{k \,  \le \, c(\log X)^3}(2^k-1) \ \ \le \ \ 2^{E}\,,
\end{equation}
with
\begin{equation}\label{fewbad4}
    E \ \ = \ \ \sum_{k \,  \le \, c(\log X)^3} k \ \ \le \ \ \f 12 c^2 (\log X)^6\,,
\end{equation}
implying (\ref{sizeofB}).
\end{proof}
In fact this proof shows something slightly stronger, that the bound (\ref{sizeofB}) holds for the number of primes at least $X$ whose multiplicative order is bounded by $c(\log X)^3$.

\vspace{1 cm}

\noindent Addresses:

\vspace{.3 cm}

\noindent Stephen D. Miller
\\ Department of Mathematics \\ 110 Frelinghuysen Road \\
Rutgers, The State University of New Jersey \\ Piscataway, NJ 08854
\\  {\tt miller@math.rutgers.edu}

\vspace{.3 cm}

\noindent Ramarathnam Venkatesan \\ Microsoft Research Cryptography
and Anti-Piracy Group
\\ 1 Microsoft Way \\ Redmond, WA 98052

and
\vspace{-.0cm}
\noindent \\ Cryptography, Security and Applied Mathematics Group
\\Microsoft Research India \\
Scientia - 196/36 2nd Main \\ Sadashivnagar, Bangalore 560 080, India
\\
{\tt
venkie@microsoft.com}


\begin{thebibliography}{99}

\bib{hellmanpohlig}{article}{
   author={Pohlig, Stephen C.},
   author={Hellman, Martin E.},
   title={An improved algorithm for computing logarithms over ${\rm GF}(p)$
   and its cryptographic significance},
   journal={IEEE Trans. Information Theory},
   volume={IT-24},
   date={1978},
   number={1},
   pages={106--110}
}

\bib{horwen}{article}{
    author={Horwitz, Jeremy},
    author={Venkatesan, Ramarathnam},
     title={Random Cayley digraphs and the discrete logarithm},
 booktitle={Algorithmic number theory (Sydney, 2002)},
    series={Lecture Notes in Comput. Sci.},
    volume={2369},
     pages={416\ndash 430},
 publisher={Springer},
     place={Berlin},
      date={2002},
}

\bib{kmpt1}{article}{author={Kim, J.-H.},author={Montenegro, R.},author={  Tetali, P.},title={Near Optimal Bounds for Collision in Pollard Rho for Discrete Log},booktitle={Proceedings of the 48th Annual IEEE Symposium on Foundations of Computer Science (FOCS 2007)},pages={215\ndash 223},year={2007}}


\bib{kmpt2}{article}{author={Kim, J.-H.},author={Montenegro, R.},author={Peres, Y.},author={Tetali, P.},title={A Birthday Paradox for Markov chains, with an optimal bound for collision in Pollard Rho for Discrete Logarithm},
conference={
      title={Algorithmic number theory},
   },
   book={
      series={Lecture Notes in Comput. Sci.},
      volume={5011},
      publisher={Springer},
      place={Berlin},
   },
   date={2008},
   pages={402--415}
}

\bib{prho}{article}{
   author={Miller, Stephen D.},
   author={Venkatesan, Ramarathnam},
   title={Spectral analysis of Pollard rho collisions},
   conference={
      title={Algorithmic number theory},
   },
   book={
      series={Lecture Notes in Comput. Sci.},
      volume={4076},
      publisher={Springer},
      place={Berlin},
   },
   date={2006},
   pages={573--581}
}
\bib{montteta}{article}{author={Ravi Montenegro},author={Prasad
Tetali}, title={Mathematical Aspects of Mixing Times in Markov
Chains}, booktitle={Foundations and Trends in Theoretical Computer
Science}, date={2006}}



\bib{shoup}{article}{
    author={Shoup, Victor},
     title={Lower bounds for discrete logarithms and related problems},
booktitle={Advances in cryptology---EUROCRYPT '97 (Konstanz)},
    series={Lecture Notes in Comput. Sci.},
    volume={1233},
     pages={256\ndash 266},
publisher={Springer},
     place={Berlin},
      date={1997},
    note={Updated version at
\url{http://www.shoup.net/papers/dlbounds1.pdf}} }




\end{thebibliography}
\end{document}